\documentclass[a4paper,11pt,reqno]{amsart}

\usepackage[margin=1in,includehead,includefoot]{geometry}
\usepackage[utf8]{inputenc}
\usepackage[T1]{fontenc}
\usepackage{lmodern,microtype}
\usepackage{hyperref}
\usepackage{url}
\usepackage{booktabs}
\usepackage{mathtools,amssymb,amsthm,amsmath}
\usepackage[foot]{amsaddr}
\usepackage{graphicx, color}
\usepackage[margin=5mm]{caption}
\usepackage{enumitem}
\usepackage[textsize=footnotesize]{todonotes}
\usepackage{wrapfig}
\usepackage{mycommands}

\graphicspath{{./graphics/}}

\newtheorem{theorem}{Theorem}
\newtheorem{lemma}[theorem]{Lemma}

\hypersetup{
  pdftitle={Hamiltonicity of Schrijver graphs and stable Kneser graphs},
  pdfauthor={Torsten M\"utze, Namrata}
}

\title{Hamiltonicity of Schrijver graphs \\ and stable Kneser graphs}

\author{Torsten M\"utze}
\address[Torsten M\"utze]{Department of Computer Science, University of Warwick, United Kingdom \& Department of Theoretical Computer Science and Mathematical Logic, Charles University, Prague, Czech Republic}
\email{torsten.mutze@warwick.ac.uk}

\author{Namrata}
\address[Namrata]{Department of Computer Science, University of Warwick, United Kingdom}
\email{namrata@warwick.ac.uk}

\thanks{This work was supported by Czech Science Foundation grant GA~22-15272S. Both authors participated in the workshop `Combinatorics, Algorithms and Geometry' in March 2024, which was funded by German Science Foundation grant~522790373.}

\begin{document}

\begin{abstract}
For integers~$k\geq 1$ and~$n\geq 2k+1$, the Schrijver graph~$S(n,k)$ has as vertices all $k$-element subsets of~$[n]:=\{1,2,\ldots,n\}$ that contain no two cyclically adjacent elements, and an edge between any two disjoint sets.
More generally, for integers~$k\geq 1$, $s\geq 2$, and $n \geq sk+1$, the $s$-stable Kneser graph~$S(n,k,s)$ has as vertices all $k$-element subsets of~$[n]$ in which any two elements are in cyclical distance at least~$s$.
We prove that all the graphs~$S(n,k,s)$, in particular Schrijver graphs~$S(n,k)=S(n,k,2)$, admit a Hamilton cycle that can be computed in time~$\cO(n)$ per generated vertex.
\end{abstract}

\maketitle

\section{Introduction}
\label{sec:schrijver}

For integers~$k\geq 1$ and~$n\geq 2k+1$, the \defi{Kneser graph~$K(n,k)$} has as vertices all $k$-element subsets of~$[n]:=\{1,2,\ldots,n\}$, and an edge between any two sets~$A$ and~$B$ that are disjoint, i.e., $A\cap B=\emptyset$.
Kneser graphs were popularized by Lov\'asz~\cite{MR514625} in his celebrated proof of Kneser's conjecture, where he proved that the chromatic number of~$K(n,k)$ equals~$n-2(k-1)$, using topological arguments.
Schrijver~\cite{MR512648} considered an induced subgraph of~$K(n,k)$, obtained by considering as vertices only the $k$-element sets that contain no two cyclically adjacent numbers, i.e., for any vertex~$A\seq [n]$ and any two elements~$i,j\in A$ with $i<j$, we require that $j-i\geq 2$ and~$j-i\neq n-1$.
This graph is referred as \defi{Schrijver graph}~$S(n,k)$, and Schrijver showed that it is vertex-critical, i.e., it has the same chromatic number as the Kneser graph~$K(n,k)$, but removing any of its vertices will decrease the chromatic number.
Both Kneser and Schrijver graphs are generalized by~\defi{$s$-stable Kneser graphs~$S(n,k,s)$}.
Those are defined for integers~$k\geq 1$, $s\geq 1$, and $n \geq \max\{2,s\}k+1$, as the induced subgraph of~$K(n,k)$ obtained by considering as vertices only the $k$-element sets that contain no two numbers in cyclical distance less than~$s$, i.e., for any vertex~$A\seq [n]$ and any two elements~$i,j\in A$ with $i<j$, we require that $j-i\geq s$ and~$j-i\notin \{n-1,n-2,\ldots,n-s+1\}$.
Note that $S(n,k,1)$ is precisely the Kneser graph~$K(n,k)$, whereas $S(n,k,2)$ is precisely the Schrijver graph~$S(n,k)$.
Meunier~\cite[Conj.~2]{MR2793613} conjectured that the chromatic number of~$S(n,k,s)$ equals~$n-s(k-1)$ when $s\geq 2$, and some progress in this direction has been made~\cite{jonsson:2012,MR3346141}.

In this paper, we consider the problem of finding Hamilton cycles in Schrijver graphs, and more generally, in $s$-stable Kneser graphs.
For Kneser graphs~$K(n,k)$, this problem was recently settled affirmatively by Merino, M\"utze and Namrata~\cite{MR4617441}.
Specifically, we showed that all Kneser graphs~$K(n,k)$, $n\geq 2k+1$, admit a Hamilton cycle, unless $(n,k)=(5,2)$, which is the exceptional Petersen graph.
This result followed a long line of work (see, e.g.~\cite{MR1883565,MR1999733,MR2836824,MR4273468}), and it resolved a problem that was open for 50~years.

In this work, we apply the methods developed in~\cite{MR4617441} to Schrijver graphs and $s$-stable Kneser graphs, which yields the following new results.

\begin{theorem}
\label{thm:ham}
For any $k\geq 1$, $s\geq 2$ and $n\geq sk+1$, the $s$-stable Schrijver graph~$S(n,k,s)$ has a Hamilton cycle.
In particular, the Schrijver graph~$S(n,k)=S(n,k,2)$ has a Hamilton cycle.
\end{theorem}

We also provide an efficient algorithm for computing these cycles.

\begin{theorem}
\label{thm:algo}
Let $k\geq 1$, $s\geq 2$ and $n\geq sk+1$.
There is an algorithm for computing a Hamilton cycle in~$S(n,k,s)$ that takes time~$\cO(n)$ to compute the next vertex on the cycle.
\end{theorem}

The initialization time and memory required by our algorithm is also~$\cO(n)$.
The assumption $s\geq 2$ in Theorem~\ref{thm:algo} is important.
In fact, we do not know any algorithm with running time polynomial in~$n$ and~$k$ for computing a Hamilton cycle in~$S(n,k,1)=K(n,k)$, even though by the aforementioned results~\cite{MR4617441}, $K(n,k)$ is known to have a Hamilton cycle (unless $(n,k)=(5,2)$).

We implemented the algorithm given by Theorem~\ref{thm:algo} in~C++, and made it available for experimentation and download on the Combinatorial Object Server website~\cite{cos_schrijver}.

After submitting this manuscript, we learnt that Theorem~\ref{thm:ham} was proved independently by Ledezma and Pastine~\cite{ledezma-pastine:24}.

\subsection{Outline of this paper}

In Sections~\ref{sec:proof-thm} and~\ref{sec:proof-algo} we present the proofs of Theorem~\ref{thm:ham} and~\ref{thm:algo}, respectively, for the special case of Schrijver graphs~$S(n,k)=S(n,k,2)$.
With minimal adaptations, these proofs can be adjusted to also work for the general case of $s$-stable Kneser graphs for any $s\geq 2$.
We explain those adaptations in Section~\ref{sec:general}.

\section{Proof of Theorem~\ref{thm:ham} for Schrijver graphs~$S(n,k)$}
\label{sec:proof-thm}

As mentioned before, we apply the methods developed in~\cite{MR4617441} to construct a Hamilton cycle in the Schrijver graph~$S(n,k)$.
The construction proceeds in two steps:
In the first step we construct a \defi{cycle factor} in the graph, i.e., a collection of disjoint cycles that together visit all vertices.
In the second step we connect the cycles of the factor to a single Hamilton cycle.

\subsection{Cycle factor construction and gliders}

\subsubsection{Preliminaries}

We interpret vertices of the Schrijver graph~$S(n,k)$ as bitstrings by considering the corresponding characteristic vectors.
For every $k$-element subset~$A$ of~$[n]$, this is a bitstring of length~$n$ with exactly $k$ many 1s at the positions corresponding to the elements of~$A$.
For example, the vertex~$A=\{2,7,9\}$ of $S(9,3)$ is represented by the bitstring~$x=010000101$.
Throughout this paper, we consider the indices into these bitstrings modulo~$n$, i.e., $x_{n+i}=x_i$.
By definition of Schrijver graphs, no two 1s in $x$ are next to each other (considered cyclically), or in other words, every 1-bit is surrounded by 0-bits.
We write~$X_{n,k}$ for the set of all bitstrings that encode vertices of~$S(n,k)$.
A straightforward counting argument shows that~$|X_{n,k}|=\binom{n-k+1}{k}-\binom{n-k-1}{k-2}=\frac{n}{n-k}\binom{n-k}{k}$ (see~\cite{MR9006}).
For any bitstring~$x$ and any integer~$r\geq 0$, we write $x^r$ for the $r$-fold concatenation of~$x$ with itself.

For two strings~$x$ and~$y$ over any alphabet, we write $x\prec y$ if $x$ is lexicographically strictly smaller than~$y$, and for a set of strings~$X$ we write $\lexmin X$ for the lexicographically least string in~$X$, i.e., for the unique string~$x\in X$ that satisfies $x\prec y$ for all $y\in X\setminus\{x\}$.

\subsubsection{Cycle factor construction}
\label{sec:factor}

For a bitstring~$x\in X_{n,k}$, we write $\sigma(x)$ for the string obtained from~$x$ by a cyclic right shift by 1~position.
We define the \defi{necklace of~$x$} by
\begin{equation*}
\neck{x} := \{\sigma^i(x)\mid i\geq 0\},
\end{equation*}
i.e., the necklace~$\neck{x}$ is the set of all cyclic shifts of~$x$.

As $x\in X_{n,k}$ has no two 1s next to each other, $(x,\sigma(x))$ is an edge in the Schrijver graph.
Consequently, for a bitstring~$x$ and~$r:=|\neck{x}|$, the sequence
\begin{subequations}
\label{eq:Cnk}
\begin{equation}
\label{eq:Cx}
C(x):=\big(x,\sigma(x),\sigma^2(x),...\sigma^{r-1}(x)\big)
\end{equation}
defines a cycle in~$S(n,k)$; see the top part of Figure~\ref{fig:glider}.
We thus obtain a cycle factor in~$S(n,k)$ by
\begin{equation}
\label{eq:cyc-fac}
\cC_{n,k}:=\{C(x) \mid x\in X_{n,k}\}.
\end{equation}
\end{subequations}

\subsubsection{Matched and unmatched bits}
\label{sec:matching}

The operation~$x\mapsto \sigma(x)$ can also be described as follows:
To create the string~$\sigma(x)$ from~$x$, every 1-bit at some position~$i$ in~$x$ is transposed with the 0-bit immediately to the right of it, i.e., at position~$i+1$.
Given $x\in X_{n,k}$, we refer to the bits at positions $\mu(x):=\{i\mid x_i=1 \text{ or } x_{i-1}=1\}\seq [n]$ as \defi{matched}, and to the bits at positions~$\ol{\mu}(x):=[n]\setminus \mu(x)=\{i\mid x_i=0 \text{ and } x_{i-1}=0\}$ as \defi{unmatched}.
The motivation for this terminology will become clear momentarily.
We denote unmatched bits in~$x$ as~$\hyph$.
For example, we write $x=010001001010101=010\hyph\hyph10\hyph 1010101$ and we have $\mu(x)=\{1,2,3,6,7,9,10,11,12,13,14,15\}$ and $\ol{\mu}(x)=\{4,5,8\}$.
Note that $|\mu(x)|=2k$ and therefore~$|\ol{\mu}(x)|=n-2k$, so $x$ has exactly $n-2k$ many~$\hyph$s.
On the right hand side of Figure~\ref{fig:glider}, matched bits are colored, and unmatched bits are white.

We refer to any maximal (cyclic) substring of matched bits in~$x\in X_{n,k}$ as a \defi{block} and to any maximal (cyclic) substring of unmatched bits as a \defi{gap}.
The \defi{length} of a block or gap is the number of bits belonging to the block or gap, respectively.
Clearly, the length of a block is always even.
For example, the bitstring~$x$ from before has two blocks of lengths~2 and~10, and two gaps of lengths~2 and~1.

\subsubsection{Gliders}

We consider every substring~$10$ in~$x$ as an entity called a \defi{glider}.
Specifically, the set of all gliders in~$x$ is the set
\begin{equation}
\label{eq:Gammax}
\Gamma(x):=\big\{(i,i+1)\mid x_i=1 \text{ and } x_{i+1}=0 \big\}.
\end{equation}
For example, for $x=010\hyph\hyph10\hyph 1010101$ from before we have~$\Gamma(x)=\{(2,3),(6,7),(9,10),(11,12),\allowbreak (13,14),(15,1)\}$.

\begin{figure}
\includegraphics{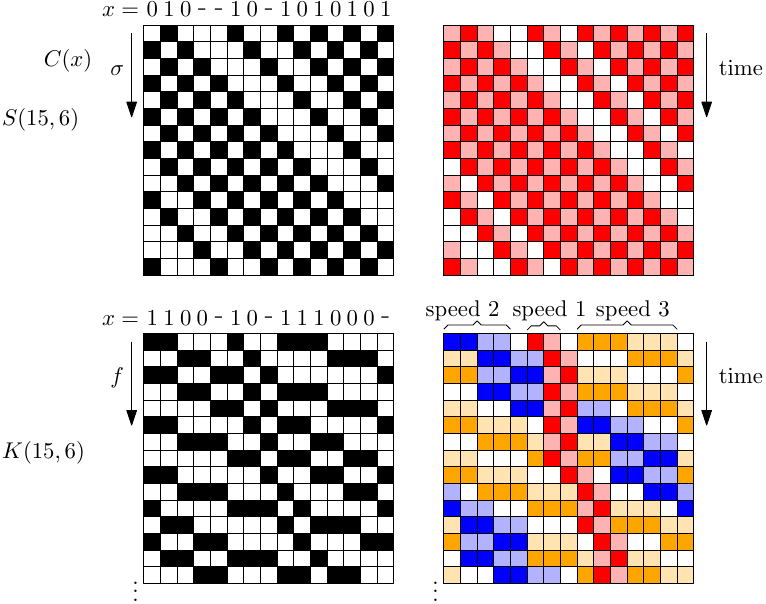}
\caption{A cycle in the Schrijver graph~$S(15,6)$ constructed from the mapping~$\sigma$ (top), and a cycle in the Kneser graph~$K(15,6)$ constructed from the more general mapping~$f$.
The cycle in~$S(15,6)$ is shown completely, whereas only the first 15 vertices of the cycle in~$K(15,6)$ are displayed.
On the left, each vertex is represented by a bitstring, with 1-bits colored black and 0-bits colored white.
The vertices are printed from top to bottom in the order of the cycle.
The right-hand side shows the interpretation of certain groups of bits as gliders, and their movement over time.
Matched bits belonging to the same glider are given the same color, with the opaque filling given to 1-bits, and the transparent filling given to 0-bits.
}
\label{fig:glider}
\end{figure}

The motivation for the word `glider' becomes clear when comparing the gliders in~$x$ with those in~$\sigma(x)$, and when interpreting one application of~$\sigma$ (i.e., moving one step along the cycle~$C(x)$) as one unit of time moving forward.
Indeed, if $(i,i+1)\in \Gamma(x)$, then we have $(i+1,i+2)\in \Gamma(\sigma(x))$, i.e., a glider that occupies positions~$i$ and~$i+1$ in~$x$ moves to positions~$i+1$ and~$i+2$ after one time step; see Figure~\ref{fig:glider}.
We can thus view these pairs of matched bits as entities that change position over time, namely they move one position to the right with each time step, i.e., their `speed' is~1.

\subsubsection{Connection to parenthesis matching in Kneser graphs}

The aforementioned mapping~$x\mapsto \sigma(x)$ can be generalized to yield a more general mapping~$x\mapsto f(x)$ that recovers precisely the construction of a cycle factor in~$K(n,k)$ used in~\cite{MR4617441}, which agrees with the one before on the Schrijver subgraph~$S(n,k)\seq K(n,k)$, as follows:
To create the string~$f(x)$ from~$x$, where $x$ now is an arbitrary bitstring of length~$n$ with $k$ many 1s, every 1-bit at some position~$i$ in~$x$ is transposed with the 0-bit that terminates the shortest possible substring at positions~$i+1,i+2,\ldots$ that has strictly more~0s than~1s; see the bottom part of Figure~\ref{fig:glider}.
An alternative and equivalent definition of~$f$ is to interpret the 1s in~$x$ as opening brackets and the 0s as closing brackets, to match pairs of opening and closing brackets in the natural way (leaving $n-2k$ closing brackets unmatched), and to transpose every bit with the corresponding matched partner to obtain~$f(x)$.
This idea, pioneered by Greene and Kleitman~\cite{MR0389608} in another context, is often referred to as \defi{parenthesis matching}, which explains our choice of terminology for matched and unmatched bits before.
In fact, the paper~\cite{MR4617441} describes a more general partition of groups of matched bits into gliders that works in general Kneser graphs, and then different gliders may move at different speeds; see the bottom part of Figure~\ref{fig:glider}.
For our discussion of Schrijver graphs this more general definition of gliders is not needed.

\subsection{Gluing the cycles together}
\label{sec:gluing}

We now show how to connect the cycles in the factor~$\cC_{n,k}$ defined in~\eqref{eq:Cnk} together to a single Hamilton cycle of~$S(n,k)$.
Each joining step proceeds by gluing together two cycles from the factor to one cycle via a 4-cycle.

\subsubsection{Connectors}
\label{sec:connectors}

Given two bitstrings~$x,y\in X_{n,k}$, we refer to the pair~$(x,y)$ as a \defi{connector}, if $x$ and~$\sigma^{-1}(y)$ agree in all but three positions~$i,i+1,i+2$ such that $x_{i,i+1,i+2}=10\hyph$ and $\sigma^{-1}(y)_{i,i+1,i+2}=\hyph 10$, i.e., we have
\begin{equation}
\label{eq:conn}
\begin{bmatrix}
x \\ \sigma^{-1}(y)
\end{bmatrix}
=
\begin{bmatrix}
\cdots1\,0 \, \hyph \cdots\\
\cdots\hyph\, 1\,0\cdots
\end{bmatrix},
\end{equation}
where the vertically aligned $\cdots$ stand for substrings in which~$x$ and~$\sigma^{-1}(y)$ agree.
In other words, $x$ and~$\sigma^{-1}(y)$ have the same gliders, except that the rightmost glider in a block of~$x$ is `pushed' to the right by one position, i.e., it is transposed with the unmatched bit to its right.

We write~$\cX_{n,k}$ for the set of all connectors.
Note that a bitstring~$z\in\cX_{n,k}$ may be part of several different connectors.
Specifically, if there are~$p$ gaps in~$z$, then $z$ is contained in exactly~$p$ connectors of the form~$(z,y)$ and $p$ connectors of the form~$(x,z)$.

\begin{lemma}
\label{lem:connect}
For any connector~$(x,y)\in \cX_{n,k}$, the sequence $C_4(x,y):=(x,\sigma(x),\sigma(y),y)$ is a 4-cycle in the Schrijver graph~$S(n,k)$.
\end{lemma}

\begin{proof}
We already know that $(x,\sigma(x))$ and~$(y,\sigma(y))$ are edges in~$S(n,k)$.
Therefore, it suffices to prove that~$(x,y)$ is an edge in~$S(n,k)$, as this will imply that~$(\sigma(x),\sigma(y))$ is also an edge.
Let $i,i+1,i+2$ be such that~$x_{i,i+1,i+2}=10\hyph$ and~$\sigma^{-1}(y)_{i,i+1,i+2}=\hyph 10$.
From~\eqref{eq:conn} we see that~$y$ agrees with $\sigma(x)$ in all positions but~$i+1,i+2,i+3$, specifically that $\sigma(x)_{i+1,i+2,i+3}=10\hyph$ and $y_{i+1,i+2,i+3}=\hyph 10$.
As $\sigma(x)$ and~$x$ have no 1s at the same position, and~$x_{i+2}=\hyph$ (this is the position where $y$ has a~1), we obtain that~$y$ and~$x$ have no 1s at the same position, i.e., $(x,y)$ is indeed an edge in~$S(n,k)$.
\end{proof}

Observe that if $(x,y)\in \cX_{n,k}$ and $C(x)$ and~$C(y)$ are two distinct cycles in the factor~$\cC_{n,k}$ defined in~\eqref{eq:Cnk}, then the symmetric difference of the edge sets~$\big(C(x) \cup  C(y)\big)\Delta C_4(x,y)$ is a single cycle in~$S(n,k)$ on the same vertex set as~$ C(x) \cup C(y)$, i.e., the 4-cycle `glues' two cycles from the factor together to a single cycle.

We define
\begin{equation}
\label{eq:Ynk}
Y_{n,k}:=\{x\in X_{n,k}\mid x_1=\hyph \text{ and } x_n=0\}=\{x\in X_{n,k}\mid 1\in \ol{\mu}(x) \text{ and } n\in \mu(x)\}.
\end{equation}
These are vertices of~$S(n,k)$ in which the first bit is unmatched and the last bit is matched (specifically, a matched 0-bit).
Any $x\in Y_{n,k}$ has the form
\begin{subequations}
\label{eq:Lxneck}
\begin{equation}
\label{eq:x}
x=\hyph^{\gamma_1}(10)^{\beta_1/2}\;\hyph^{\gamma_2}(10)^{\beta_2/2}\;\cdots \;\hyph^{\gamma_r}(10)^{\beta_r/2}
\end{equation}
for integers $r\geq 1$, $\gamma_1,\ldots,\gamma_r>0$, and even $\beta_1,\ldots,\beta_r>0$.
In words, $\gamma_i$ is the length of the $i$th gap in~$x$, and $\beta_i$ is the length of the $i$th block in~$x$.
We then define
\begin{equation}
\label{eq:Lx}
L(x):=(-\beta_r,\gamma_r,-\beta_{r-1},\gamma_{r-1},\ldots,-\beta_2,\gamma_2,-\beta_1,\gamma_1),
\end{equation}
i.e., $L(x)$ records the lengths of blocks and gaps in~$x$ from \emph{right to left}, and the lengths of blocks are given a negative sign.
For example, for the bitstring~$x=\hyph \hyph \hyph \hyph 10101010 \hyph \hyph \hyph 10 \hyph 1010 \hyph \hyph 101010\in Y_{30,10}$ we have $L(x)=(-6,2,-4,1,-2,3,-8,4)$.

For any $x\in X_{n,k}$ we define $L(\neck{x})$ as the lexicographically smallest string~$L(x')$ over all $x'\in \neck{x}\cap Y_{n,k}$, i.e.,
\begin{equation}
L(\neck{x}):=\lexmin \big\{L(x') \mid x'\in \neck{x}\cap Y_{n,k}\big\}.
\end{equation}
\end{subequations}
In words, we consider all cyclic shifts~$x'$ of~$x$ that belong to the set~$Y_{n,k}$, and take the lexicographically least string~$L(x')$ among them.
Clearly, cyclic shifts of~$x\in Y_{n,k}$ yield different cyclic shifts of~$L(x)$, all starting with a negative number.
In the example from before, we obtain~$L(\neck{x})=(-8,4,-6,2,-4,1,-2,3)$.
Informally, for a given~$x\in X_{n,k}$, the definition~$L(\neck{x})$ locates the (cyclically) longest possible block in~$x$, preceded by the shortest possible gap, preceded by the longest possible block, preceded by the shortest possible gap, etc., which motivates the negative signs in~\eqref{eq:Lx} and the right-to-left reading.

We will also need the following variant of the definition~\eqref{eq:Lx} from before.
For any $x\in X_{n,k}$ and any position~$p\in[n]$ with $x_{p,p+1}=0\hyph$, i.e., the $p$th bit of~$x$ is matched and the $(p+1)$st bit is unmatched, we have $\sigma^{-p}(x)\in Y_{n,k}$ and may thus define
\begin{equation}
\label{eq:Lxp}
L(x,p):=L(\sigma^{-p}(x)).
\end{equation}
In words, this is the sequence of lengths of blocks and gaps in~$x$ starting from position~$p$ read in the right-to-left direction (with negative signs for block lengths).

\subsubsection{Auxiliary graph}

Our strategy is to join the cycles of the factor~$\cC_{n,k}$ by repeatedly gluing pairs of them together via connectors, as described before.
For any set of connectors~$\cU\seq\cX_{n,k}$ we define a graph~$\cH_{n,k}[\cU]$ as follows:
The nodes of~$\cH_{n,k}[\cU]$ are the cycles of the factor~$\cC_{n,k}$.
Furthermore, for any connector~$(x,y)\in\cU$ for which $C(x)$ and~$C(y)$ are distinct cycles, we add an edge that connects~$C(x)$ and~$C(y)$ to the graph~$\cH_{n,k}[\cU]$.
To obtain a Hamilton cycle in~$S(n,k)$, we require two things:
The graph~$\cH_{n,k}[\cU]$ must be connected.
Furthermore, we require that any two of the 4-cycles $C_4(x,y)$ and~$C_4(x',y')$ with $(x,y),(x',y')\in \cU$ used for the joining are edge-disjoint.
To ensure the second property, it is enough to guarantee that~$C_4(x,y)$ and~$C_4(x',y')$ do not have an edge in common on one of the cycles of the factor~$\cC_{n,k}$.
This condition is equivalent to~$\{x,y\}\cap \{x',y'\}=\emptyset$, and if it holds then we call the connectors~$(x,y)$ and~$(x',y')$ \defi{disjoint}.
We therefore require a set of pairwise disjoint connectors~$\cU\seq\cX_{n,k}$ such that~$\cH_{n,k}[\cU]$ is a connected graph.
Note the following subtlety:
There may be two connectors~$(x,y),(x',y')\in\cU$ with $x'=\sigma(x)$ and~$y'=\sigma(y)$, which by our definition are disjoint, even though the 4-cycles~$C_4(x,y)$ and~$C_4(x',y')$ share the edge~$(x',y')=(\sigma(x),\sigma(y))$.
However, in the joining process, only one of the two 4-cycles will be used, as they connect the same pair of cycles~$C(x)=C(x')$ and~$C(y)=C(y')$.

A set of connectors with the aforementioned properties can be defined explicitly for any fixed integer~$p\in[n]$ by
\begin{equation}
\label{eq:Up}
\cU_p := \big\{(x,y)\in \cX_{n,k}\mid x_{p,p+1,p+2}=10\hyph \text{ and } \sigma^{-1}(y)_{p,p+1,p+2}=\hyph 10\big\};
\end{equation}
recall~\eqref{eq:conn}.
In words, these are connectors in which the `pushing' of a glider occurs at position~$p$.

\begin{lemma}
\label{lem:HU}
For any~$k\geq 1$, $n\geq 2k+1$, and $p\in[n]$, let~$\cU_p\seq\cX_{n,k}$ be the set of connectors defined in~\eqref{eq:Up}.
The connectors in~$\cU_p$ are pairwise disjoint, and $\cH_{n,k}[\cU_p]$ is a connected graph.
\end{lemma}

\begin{proof}
Consider two connectors~$(x,y),(x',y')\in \cU_p$ with $x\neq x'$ which implies that~$y\neq y'$.
Clearly, we have $y\neq x$ and $y'\neq x'$.
From~\eqref{eq:Up} we also see that $x_{p+2}=\hyph$ and $\sigma^{-1}(y')_{p+1}=y'_{p+2}=1$, implying that $x\neq y'$, and symmetrically that $x'_{p+2}=\hyph$ and $\sigma^{-1}(y)_{p+1}=y_{p+2}=1$, implying that $x'\neq y$.
This proves that $\{x,y\}\cap \{x',y'\}=\emptyset$, as desired.

\begin{figure}
\makebox[0cm]{ % artificial box to center the picture
\includegraphics{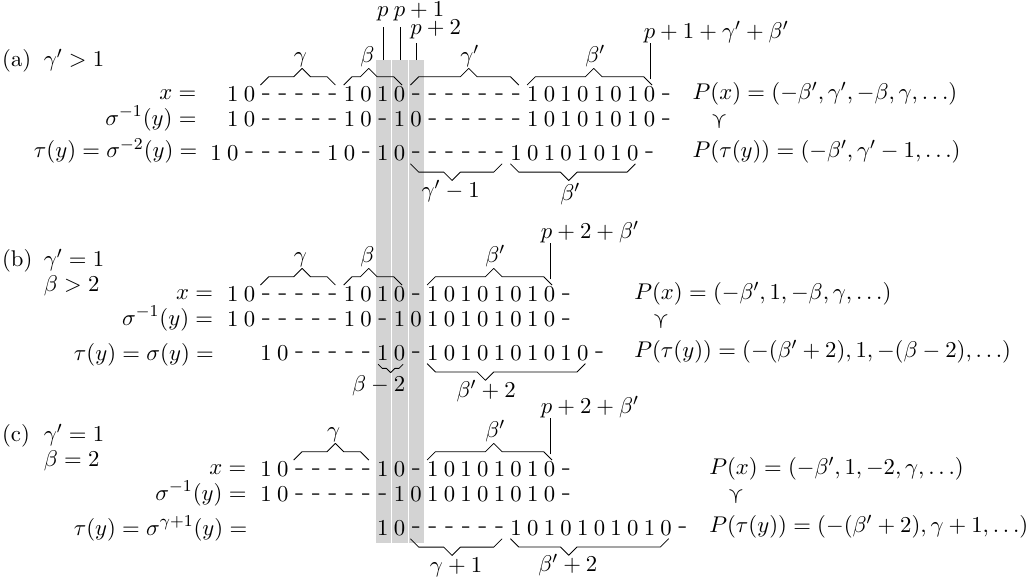}
}
\caption{Illustration of the proof of Lemma~\ref{lem:HU}.}
\label{fig:push}
\end{figure}

It remains to prove that~$\cH_{n,k}[\cU_p]$ is connected.
We say that a vertex~$x\in X_{n,k}$ is \defi{connectable}, if $x_{p,p+1,p+2}=10\hyph$.
By~\eqref{eq:Up}, this means that there is a unique corresponding connector~$(x,y)\in \cU_p$.
For any connectable vertex~$x$ with at least two blocks, we let $\beta$ be the length of the block ending at position~$p+1$ in~$x$, $\gamma$ be the length of the gap to the left of it, $\gamma'$ be the length of the gap to the right of it, and we let $\beta'$ be the length of the block to the right of this gap; see Figure~\ref{fig:push}~(a).
We define
\begin{equation}
\label{eq:Px}
P(x):=L(x,p+1+\gamma'+\beta')
\end{equation}
with $L(x,p)$ as defined in~\eqref{eq:Lxp}.
In words, $P(x)$ is obtained by recording the lengths of blocks and gaps in~$x$ from right to left, starting at the first block strictly to the right of position~$p+1$ in~$x$, where block lengths receive an additional negative sign.
We thus have $P(x)=(-\beta',\gamma',-\beta,\gamma,\ldots)$.
We also define the connectable vertex
\begin{equation}
\label{eq:tau}
\tau(y):=\begin{cases}
\sigma^{-2}(y) & \text{if } \gamma'>1, \\
\sigma(y) & \text{if } \gamma'=1 \text{ and } \beta>2, \\
\sigma^{\gamma+1}(y) & \text{if } \gamma'=1 \text{ and } \beta=2;
\end{cases}
\end{equation}
see Figure~\ref{fig:push}.
It satisfies
\begin{equation}
\label{eq:Ptau}
P(\tau(y))=\begin{cases}
(-\beta',\gamma'-1,\ldots) & \text{if } \gamma'>1, \\
(-(\beta'+2),\ldots) & \text{if } \gamma'=1,
\end{cases}
\end{equation}
implying that $P(\tau(y))\prec P(x)$, i.e., $P(\tau(y))$ is lexicographically smaller than~$P(x)$.
This shows that we can move in the graph~$\cH_{n,k}[\cU_p]$ along lexicographically decreasing connectors, to the lexicographically least connectable vertex, which lies in the cycle $C(\hyph^{n-2k}(10)^k)$, i.e., it has only a single block and a single gap (all gliders are adjacent).
This proves that $\cH_{n,k}[\cU_p]$ is indeed connected, which completes the proof.
\end{proof}

We are now ready to present the proof of Theorem~\ref{thm:ham} for Schrijver graphs~$S(n,k)$.

\begin{proof}[Proof of Theorem~\ref{thm:ham} for Schrijver graphs~$S(n,k)=S(n,k,2)$]
By Lemma~\ref{lem:HU}, the set~$\cU_p\seq X_{n,k}$ of connectors defined in~\eqref{eq:Up} for any fixed value of~$p\in[n]$ has the property that the connectors in~$\cU_p$ are pairwise disjoint, and $\cH_{n,k}[\cU_p]$ is a connected graph.
Let~$\cT$ be a minimal subset of~$\cU_p$ such that~$\cH_{n,k}[\cT]$ is connected, i.e., this graph is a spanning tree.
As the connectors in~$\cU_p$ are pairwise disjoint, no two of the 4-cycles $C_4(x,y)$, $(x,y)\in\cT$, defined in Lemma~\ref{lem:connect} have an edge in common with one of the cycles of the factor~$\cC_{n,k}$ in~$S(n,k)$ defined in~\eqref{eq:Cnk}.
Furthermore, by the minimality of~$\cT$, no two of these 4-cycles have an edge in common that does not lie on one of the cycles of the factor (otherwise, $\cT$ would contain two edges connecting the same pair of nodes), so these 4-cycles are pairwise edge-disjoint.
Consequently, as $\cH_{n,k}[\cT]$ is connected, the symmetric difference of the edge set of the cycle factor~$\cC_{n,k}$ with the 4-cycles~$C_4(x,y)$, $(x,y)\in\cT$, is a Hamilton cycle in~$S(n,k)$.
\end{proof}

\section{Proof of Theorem~\ref{thm:algo} for Schrijver graphs~$S(n,k)$}
\label{sec:proof-algo}

To turn the proof presented in the previous section into an efficient algorithm, we first define an explicit spanning tree in the graph~$\cH_{n,k}[\cU_p]$.
To define such a spanning tree, we select exactly one connectable vertex~$x$ on each cycle of the factor~$\cC_{n,k}$, namely the one with lexicographically least value of~$P(x)$, i.e.,
\begin{equation}
\label{eq:Tp}
\cT_p:=\{(x,y)\in\cU_p \mid x \text{ has at least 2 blocks and } P(x)=L(\neck{x})\},
\end{equation}
where $P(x)$ and $L(\neck{x})$ are defined in~\eqref{eq:Px} and~\eqref{eq:Lxneck}, respectively; see Figure~\ref{fig:lextree}.

\begin{figure}
\includegraphics{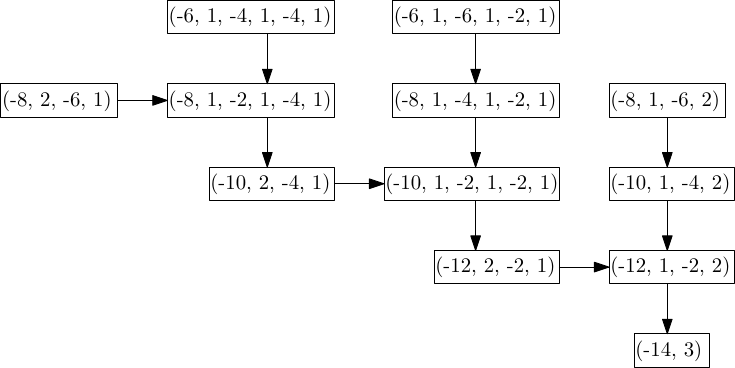}
\caption{Illustration of the tree~$\cT_p$ for $n=17$ and~$k=7$.
Each node of~$\cH_{n,k}[\cT_p]$, i.e., each cycle~$C(x)\in \cC_{n,k}$ is represented by the string $L(\neck{x})$, corresponding to the lexicographically least connectable vertex on that cycle.
The edges are oriented to point towards the lexicographically smaller string.
Horizontal and vertical edges correspond to the first and second case in~\eqref{eq:Ptau}, respectively.}
\label{fig:lextree}
\end{figure}

\begin{lemma}
\label{lem:HT}
For any~$k\geq 1$, $n\geq 2k+1$ and $p\in[n]$, let~$\cT_p\seq\cU_p\seq \cX_{n,k}$ be the set of connectors defined in~\eqref{eq:Tp}.
Then $\cH_{n,k}[\cT_p]$ is a spanning tree of~$\cH_{n,k}[\cU_p]$.
\end{lemma}

\begin{proof}
We consider the edges~$(C(x),C(y))$ of~$\cH_{n,k}[\cT_p]$ for connectors~$(x,y)\in\cT_p$ as being \emph{oriented} from~$C(x)$ to~$C(y)$.

We first show that $\cH_{n,k}[\cT_p]$ is connected.
We have argued before that for $(x,y)\in\cT_p$ we have $P(\tau(y))\prec P(x)$ (recall \eqref{eq:Ptau}).
Let $z$ be the lexicographically least connectable vertex on the cycle~$C(\tau(y))$, i.e., we have $P(z)=P(\tau(y))$ or $P(z)\prec P(\tau(y))$.
Combining these inequalities shows that $P(z)\prec P(x)$.
This shows that we can move in the graph~$\cH_{n,k}[\cT_p]$ along lexicographically decreasing connectors to the lexicographically least connectable vertex, which lies in the cycle $C(\hyph^{n-2k}(10)^k)$ (corresponding to the bottom-right vertex in Figure~\ref{fig:lextree}).
This proves that $\cH_{n,k}[\cT_p]$ is indeed connected.

It remains to argue that~$\cH_{n,k}[\cT_p]$ is acyclic.
Firstly, there can be no cycles in which the orientations of all edges agree, because of the lexicographically decreasing connectors.
Secondly, if there was a cycle in which not all edges are oriented the same, then such a cycle would contain a node of out-degree~2, but in our graph~$\cH_{n,k}[\cT_p]$, every node has out-degree exactly~1 by definition, except the node $C(\hyph^{n-2k}(10)^k)$ which has out-degree~0 (as all vertices have only a single block).
This proves that~$\cH_{n,k}[\cT_p]$ is indeed acyclic.
\end{proof}

\vspace{2mm}
\textit{Proof of Theorem~\ref{thm:algo} for Schrijver graphs~$S(n,k)=S(n,k,2)$.}
Our algorithm computes the Hamilton cycle that corresponds to the spanning tree~$\cH_{n,k}[\cT_p]$ guaranteed by Lemma~\ref{lem:HT}.
It maintains the current vertex~$x\in X_{n,k}$, and a variable~$d\in\{-1,+1\}$ that controls the current direction of movement along the cycle.
Specifically, the next vertex along the cycle~$C(x)$ is obtained by setting
\begin{equation}
\label{eq:sigmax}
x\leftarrow \sigma^d(x)=\begin{cases} \sigma(x) & \text{if } d=+1, \\ \sigma^{-1}(x) & \text{if } d=-1, \end{cases}
\end{equation}
i.e., in the first case we perform a cyclic right shift (move forward along the cycle) and in the second case the inverse operation, namely a cyclic left shift (move backward along the cycle).
Clearly, these operations take linear time~$\cO(n)$.

\begin{wrapfigure}{r}{0.3\textwidth}
\vspace{-5mm}
\includegraphics{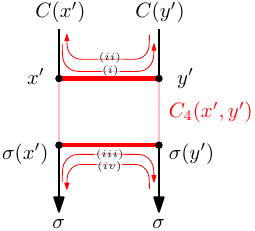}
\caption{Illustration of the instructions~(i)--(iv).}
\label{fig:algo}
\end{wrapfigure}
Furthermore, in each step, the algorithm has to check whether the current vertex~$x$ equals one of the four vertices on a cycle $C_4(x',y')=(x',\sigma(x'),\sigma(y'),y')$ for one of the connectors~$(x',y')\in\cT_p$ (recall Lemma~\ref{lem:connect}), and whether the next step according to~\eqref{eq:sigmax} would move along an edge of this cycle.
This is detected by the following conditions, followed by the corresponding alternative assignments instead of~\eqref{eq:sigmax}; see Figure~\ref{fig:algo}:
\begin{enumerate}[label=(\roman*),leftmargin=8mm]
\item if $x=x'$ and $d=+1$, then we set $x\leftarrow y'$, $d\leftarrow -1$;
\item if $x=y'$ and $d=+1$, then we set $x\leftarrow x'$, $d\leftarrow -1$;
\item if $x=\sigma(x')$ and $d=-1$, then we set $x\leftarrow \sigma(y')$, $d\leftarrow +1$;
\item if $x=\sigma(y')$ and $d=-1$, then we set $x\leftarrow \sigma(x')$, $d\leftarrow +1$.
\end{enumerate}
To decide whether $(x',y')\in\cT_p$, we need to check whether~$P(x')=M(\neck{x'})$ (recall \eqref{eq:Tp}), i.e., we need to compute the lexicographically smallest rotation of a string of length~$\leq n$, which can be done using Booth's $\cO(n)$ algorithm~\cite{MR585391}.
\qed

For more details, see the C++ implementation of our algorithm~\cite{cos_schrijver}.
Figure~\ref{fig:cycles} shows several examples of Hamilton cycles computed by our algorithm.

\begin{figure}
\begin{tabular}{ccccccc}
\raisebox{-\height}{\includegraphics{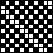}} &
\raisebox{-\height}{\includegraphics{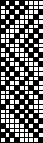}} &
\raisebox{-\height}{\includegraphics{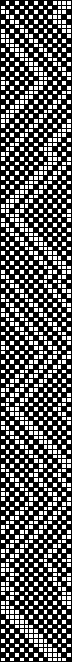}} &
\raisebox{-\height}{\includegraphics{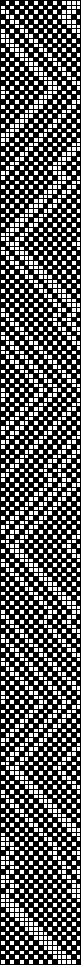}} &
\raisebox{-\height}{\includegraphics{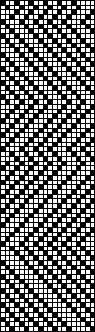}} &
\raisebox{-\height}{\includegraphics{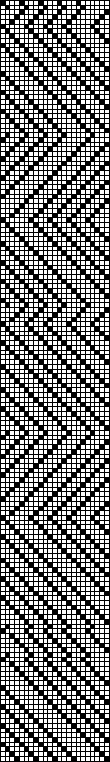}} &
\raisebox{-\height}{\includegraphics{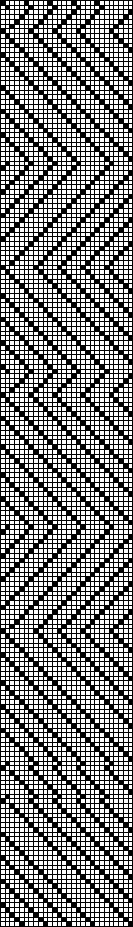}} \vspace{2mm} \\
$S(11,5)$ & $S(9,3)$ & $S(15,6)$ & $S(17,7)$ & $S(20,6,3)$ & $S(23,5,4)$ & $S(28,5,5)$
\end{tabular}
\caption{Hamilton cycles in different $s$-stable Kneser graphs computed by the algorithm from Section~\ref{sec:proof-algo}.}
\label{fig:cycles}
\end{figure}

\section{Generalizing the proofs to $s$-stable Kneser graphs~$S(n,k,s)$}
\label{sec:general}

The proofs of Theorems~\ref{thm:ham} and~\ref{thm:algo} presented in the previous sections can be generalized straightforwardly from the case $s=2$ (Schrijver graphs) to the more general case $s\geq 2$ ($s$-stable Kneser graphs), as follows.

We write $X_{n,k,s}$ for the vertices of~$S(n,k,s)$ in their bitstring representation.
It was shown in~\cite{MR2419223} that $|X_{n,k,s}|=\frac{n}{n-sk}\binom{n-sk}{k}$.
The construction of the cycle factor~$\cC_{n,k}$ given in Section~\ref{sec:factor} remains unchanged, but we slightly generalize the definition of matched bits described in Section~\ref{sec:matching}.
Each 1-bit in~$x\in X_{n,k,s}$ is matched to the next $s-1$ many 0s to the right of it, so $sk$ bits in total are matched ($k$ of them 1s, and $(s-1)k$ of them 0s) and the remaining $n-sk$ bits are unmatched.
Consequently a glider in~$x$ has the form $10^{s-1}$ and occupies $s$ positions $i,i+1,\ldots,i+s-1$ (cf.~\eqref{eq:Gammax}).
Furthermore, a connector has the form
\begin{equation*}
\begin{bmatrix}
x \\ \sigma^{-1}(y)
\end{bmatrix}
=
\begin{bmatrix}
\cdots1\,0^{s-1} \, \hyph \cdots\\
\cdots\hyph\, 1\,0^{s-1}\cdots
\end{bmatrix}
\end{equation*}
(cf.~\eqref{eq:conn}).
The decomposition~\eqref{eq:x} generalizes accordingly to
\begin{equation*}
x=\hyph^{\gamma_1}(10^{s-1})^{\beta_1/s}\;\hyph^{\gamma_2}(10^{s-1})^{\beta_2/s}\;\cdots \;\hyph^{\gamma_r}(10^{s-1})^{\beta_r/s},
\end{equation*}
and the remaining definitions in Section~\ref{sec:connectors} remain unchanged.
The definition~\eqref{eq:Up} generalizes to
\begin{equation*}
\cU_p := \big\{(x,y)\in \cX_{n,k}\mid x_{p,p+1,\ldots,p+s}=10^{s-1}\hyph \text{ and } \sigma^{-1}(y)_{p,p+1,\ldots,p+s}=\hyph 10^{s-1}\big\},
\end{equation*}
and then the proof of Lemma~\ref{lem:HU} goes through with minimal adjustments, specifically by generalizing~\eqref{eq:Px} and~\eqref{eq:tau} to
\begin{equation*}
P(x):=L(x,p+s-1+\gamma'+\beta')
\end{equation*}
and
\begin{equation*}
\tau(y):=\begin{cases}
\sigma^{-2}(y) & \text{if } \gamma'>1, \\
\sigma^{s-1}(y) & \text{if } \gamma'=1 \text{ and } \beta>s, \\
\sigma^{\gamma+s-1}(y) & \text{if } \gamma'=1 \text{ and } \beta=s,
\end{cases}
\end{equation*}
respectively.
The remainder of the proofs remains unchanged.

\bibliographystyle{alpha}
\bibliography{../refs}

\end{document}